\documentclass[reqno]{amsart}

\usepackage{geometry}
\usepackage[english]{babel}
\usepackage{latexsym, mathrsfs}
\usepackage{amsfonts, amsthm, amsmath, amssymb}
\usepackage{amsaddr}

\newtheorem{prop}{Proposition}[section]
\newtheorem{lem}[prop]{Lemma}
\newtheorem{cor}[prop]{Corollary}
\newtheorem{thm}[prop]{Theorem}
\newtheorem{conj}[prop]{Conjecture}

\theoremstyle{definition}
\newtheorem{rem}[prop]{Remark}
\newtheorem{defi}[prop]{Definition}
\newtheorem{ex}[prop]{Example}

\def\Z{\mathbb{Z}}

\def\equad{\quad \textrm{ and } \quad}

\def\H{\mathrm{H}}
\def\SMA{\mathrm{SMA}}
\def\MR{\mathrm{MR}}
\def\MRS{\mathrm{MRS}}
\def\M{\mathrm{MPF}}
\def\MS{\mathrm{MPFS}}
\def\Mo{{}^0\mathrm{MPF}}
\def\MoS{{}^0\mathrm{MPFS}}

\def\B{\mathcal{B}}
\def\E{\mathcal{E}}
\def\G{\mathcal{G}}

\numberwithin{equation}{section}

\begin{document}

\title{Magic partially filled arrays on abelian groups}

\author[Fiorenza Morini]{Fiorenza Morini}
\address{Dipartimento di Scienze Matematiche, Fisiche e Informatiche, Universit\`a di Parma,\\
Parco Area delle Scienze 53/A, 43124 Parma, Italy}
\email{fiorenza.morini@unipr.it}

\author[Marco Antonio Pellegrini]{Marco Antonio Pellegrini}
\address{Dipartimento di Matematica e Fisica, Universit\`a Cattolica del Sacro Cuore,\\
Via della Garzetta 48, 25133 Brescia, Italy}
\email{marcoantonio.pellegrini@unicatt.it}
 
\begin{abstract}
In this paper we introduce a special class of partially filled arrays.
A magic partially filled array $\M_\Omega(m,n; s,k)$ on a subset $\Omega$ of an abelian group $(\Gamma,+)$ is a  partially filled array
of size $m\times n$ with entries in $\Omega$ such that $(i)$ every $\omega \in \Omega$ appears once in the array;
$(ii)$ each row contains $s$ filled cells and each column contains $k$ filled cells;
$(iii)$ there exist (not necessarily distinct) elements $x,y\in \Gamma$ such that the sum of the elements in each row is $x$  and
the sum of the elements in each column is  $y$.
In particular, if $x=y=0_\Gamma$, we have a zero-sum magic partially filled array $\Mo_\Omega(m,n; s,k)$.
Examples of these objects are magic rectangles, $\Gamma$-magic rectangles, signed magic arrays, (integer or non integer) Heffter arrays.
Here, we give necessary and sufficient conditions for the existence of a magic rectangle with empty cells, i.e., of an $\M_\Omega(m,n;s,k)$
where $\Omega=\{1,2,\ldots,nk\}\subset\Z$. We also construct zero-sum magic partially filled arrays
when $\Omega$ is the abelian group $\Gamma$ or the set of its nonzero elements.
 \end{abstract}

\maketitle

\section{Introduction}\label{intro}

The aim of this paper is to introduce and study the following class of partially filled arrays (that is, matrices where some cells are allowed to be empty),
whose elements belong to an abelian group.

\begin{defi}\label{magic_pf}
A \emph{magic partially filled array} $\M_\Omega(m,n; s,k)$ on a subset $\Omega$ of an abelian group $(\Gamma,+)$ 
is a  partially filled array of size $m\times n$ with entries in $\Omega$ such that
\begin{itemize}
\item[{\rm (a)}] every $\omega \in \Omega$ appears once in the array;
\item[{\rm (b)}] each row contains $s$ filled cells and each column contains $k$ filled cells;
\item[{\rm (c)}] there exist (not necessarily distinct) elements $x,y\in \Gamma$ such that the sum of the elements in each row is $x$  and
the sum of the elements in each column is  $y$.
\end{itemize}
\end{defi}

Throughout this paper, we will always assume $|\Omega|>1$. So, necessary conditions for the existence of an $\M_\Omega(m,n; s,k)$ 
are $2\leq s \leq n$, $2\leq k \leq m$ and $|\Omega|=ms=nk$.

We came along with this definition considering two recent generalizations of magic rectangles.
We recall that a \emph{magic rectangle} $\MR(m,n)$ is an $m\times n$ array whose entries are the integers $1,2,\ldots,mn$,
each appearing once in such a way that the sum of the elements in each row is a constant  $x$  and the sum of the elements in each column is a constant $y$. 
These are well known objects: as shown in \cite{H1,H2}, an $\MR(m,n)$ exists if and only if $m,n>1$, $mn>4$ and $m\equiv n \pmod 2$.
In our terms, a magic rectangle is a tight $\M_\Omega(m,n;n,m)$ where $\Omega=\{1,2,\ldots,mn\}\subset (\Z,+)$.
In \cite{KL}, Khodkar and Leach considered magic rectangles with some empty cells. They gave partial results about 
the existence of an
$\MR(m,n;s,k)$, i.e., in our terminology, of an $\M_\Omega(m,n; s,k)$ where $\Omega=\{1,2,\ldots,nk\}\subset \Z$. 
On the other hand, starting from the concept of a magic square with elements on an abelian group \cite{Sun},  Cichacz studied in \cite{C15} 
the existence of a magic $m\times n$ rectangle with elements in an  abelian group $\Gamma$ of order $mn$ 
(i.e., an $\M_\Gamma(m,n;n,m)$). 

In \cite{F1,F2}, Froncek introduced the notion of magic rectangle set $\MRS(m, n; c)$.
Similarly, also Cichacz was interested in magic rectangle sets on abelian groups.

\begin{defi}\cite{C15}\label{magic_sets}
A \emph{$\Gamma$-magic rectangle set} $\MRS_\Gamma(m, n ; c )$ on an abelian group $(\Gamma,+)$ of order $m n c $ is a set of $c$ arrays of size $m\times n$, 
whose entries are elements of $\Gamma$, each appearing once, with all row sums in each rectangle equal to
a constant $x \in \Gamma$ and all column sums in each rectangle equal to a constant $y \in \Gamma$.
\end{defi}

Even if Froncek provided in \cite{F3} necessary and sufficient conditions for the existence of magic rectangle sets,
the construction of an $\MRS_\Gamma(m, n ; c )$ is, in general, still an open problem, see \cite{CH,CH2}. In particular, the following conjecture has been proposed by Cichacz and Hinc,
where $\G$ denotes the set of all finite abelian groups that either have odd order or contain more than one involution (i.e., an element of order two).

\begin{conj}\cite{CH}
Let $m,n > 1$ and $c\geq 1$. An $\MRS_\Gamma( m , n ; c )$ exists if and only if $m$ and $n$ 
are both even or $\Gamma \in\G$ and $\{ m , n \} \neq \{2\ell + 1, 2\}$.
\end{conj}

In the same spirit, we introduce the following definition.

\begin{defi}\label{magic_pf_set}
A \emph{magic partially filled array set} $\MS_\Omega(m,n; s,k; c)$ on a subset $\Omega$ of an abelian group $(\Gamma,+)$
is a set of $c$ partially filled arrays of size $m\times n$ with entries in $\Omega$  such that
\begin{itemize}
\item[{\rm (a)}] every $\omega \in \Omega$ appears  once and in a unique array;
\item[{\rm (b)}] for every array, each row contains $s$ filled cells and each column contains $k$ filled cells;
\item[{\rm (c)}] there exist (not necessarily distinct) elements $x,y\in \Gamma$ such that, for every array, the sum of the elements in each row is $x$  and
the sum of the elements in each column is  $y$.
\end{itemize}
\end{defi}

One of our main goals is the construction of magic rectangle sets with empty cells, denoted by $\MRS(m,n;s,k;c)$, 
which are nothing but $\MS_\Omega(m,n; s,k; c)$, where $\Omega=\{1,\ldots,nkc\}\subset \Z$. Note that the aforementioned (tight) magic rectangle sets $\MRS(m,n;c)$ studied by Froncek correspond to the case
$\MS_\Omega(m,n;n,m;c)$, where $\Omega=\{1,\ldots,mnc\}\subset \Z$.
We are also interested  in constructing magic partially filled arrays (sets), where the elements of each row and of each columns sum to zero.

\begin{defi}
Given a subset $\Omega$ of an abelian group $(\Gamma,+)$, we say that an $\M_\Omega(m,n; s,k)$ is a
\emph{zero-sum magic partially filled array} (and we write $\Mo_\Omega(m,n; s,k)$) 
if the elements in each row and in each column sum to  $0\in \Gamma$.
Similarly, we speak about a \emph{zero-sum magic partially filled array set} (writing $\MoS_\Omega(m,n; s,k;c)$) 
if, for every array, the elements in each row and in each column sum to  $0\in \Gamma$.
\end{defi}

Examples of $\Mo_\Omega(m,n; s,k)$  are the 
signed magic arrays, denoted by $\SMA(m,n;s,k)$ in \cite{KSW}: 
they correspond to the case
$\Omega=\{0,\pm 1, \pm 2,\ldots, \pm (nk-1)/2\}\subset \Z$ if $nk$ is odd or
$\Omega=\{\pm 1, \pm 2, \ldots, \pm nk/2\}\subset \Z$ if $nk$ is even.
Also the Heffter arrays, introduced by Archdeacon in \cite{A}, can be viewed as zero-sum magic partially filled arrays. 
 
\begin{defi}\label{def:H}
A \emph{Heffter array} $\H(m,n; s,k)$ is an $m \times n$ partially filled array with elements in the cyclic group $(\Z_{2nk+1},+)$ such that
\begin{itemize}
\item[(\rm{a})] for every $x\in \Z_{2nk+1}\setminus\{0\}$, either $x$ or $-x$ appears in the array;
\item[(\rm{b})] each row contains $s$ filled cells and each column contains $k$ filled cells;
\item[(\rm{c})] the elements in every row and column sum to $0$ in $\Z_{2nk+1}$.
\end{itemize}
\end{defi}

In \cite{CDDY} it was proved that a square Heffter array $\H(n,n;k,k)$ exists for all $n\geq k\geq 3$,
while in \cite{ABD} the authors proved the existence of a $\H(m,n;n,m)$ for all $m,n\geq 3$.
The first results about non-square Heffter arrays with empty cells have been obtained in \cite{MP,MP3}. They confirm the following.

\begin{conj}\cite[Conjecture 6.3]{A}\label{ConjArch}
Given four integers $m,n,s,k$ such that $3\leq s\leq n$, $3\leq k \leq m$ and $ms=nk$,  there exists a Heffter array $\H(m,n;s,k)$.
\end{conj}

A Heffter array $\H(m,n; s,k)$ is a $\Mo_\Omega(m,n;s,k)$ where $\Omega$ is a subset of size $nk$ of $\Z_{2nk+1}\setminus\{ 0 \}$ such that
$\Omega \cap -\Omega =\emptyset$.
In other words, $\Omega \cup -\Omega$ is a partition of  $\Z_{2nk+1}\setminus\{ 0 \}$.
In \cite{CMPP2} the authors proposed a generalization of the notion of Heffter array studying \emph{relative} Heffter arrays:
a relative Heffter array $\H_t(m,n;s,k)$ is a $\Mo_\Omega(m,n;s,k)$ where $\Omega\cup -\Omega$ is a partition of $\Z_{2nk+t}\setminus J$, with
$J$ being the subgroup of $\Z_{2nk+t}$ of size $t$. We refer to \cite{BCDY, CDY, CMPP1, CPP, DM, DP, MP3} for results about the existence and the properties of these objects.
Further generalizations of Heffter arrays are described in \cite{CD,CDP,CP,MP2}.
\smallskip

Magic and zero-sum magic partially filled arrays are worth to be studied not only because they generalize several combinatorial objects, 
as we previously explained, but also because of their connection with magic labelings (see \cite[Sections 5.1 and 5.7]{G}).
We briefly describe how these labelings can be obtained.

A bipartite biregular graph $G(V,E)$ is a graph whose vertex set can be written as disjoint union $V=V_1\cup V_2$ with
$|V_1|=m$, $|V_2|=n$, and where each vertex of $V_1$ is connected with exactly $s$ vertices of $V_2$,
and each vertex of $V_2$ is connected with exactly $k$ vertices of $V_1$.
Now, let $M$ be an $\M_{\Omega}(m,n;s,k)$ (or, a  $\Mo_{\Omega}(m,n;s,k)$). 
We associate to  $M$ a bipartite biregular graph $\Phi_M=G(V,E)$ by taking a set $V_1$ of $m$ points, a set $V_2$ of $n$ points, and drawing an edge 
$e_{i,j}$ between the $i$-th vertex of $V_1$ and the $j$-vertex of $V_2$ if the cell $(i,j)$ of $M$ is not empty. 
Define the labeling $f_M: E\to \Omega\subseteq \Gamma$, where $f_M(e_{i,j})$ is the entry of the corresponding cell $(i,j)$
of $M$. 

According to \cite{S}, a graph is \emph{magic} if there is a labeling of its edges with distinct positive integers
such that for each vertex $v$ the sum of the labels of all edges incident with $v$ is the same
for all $v$.  A such labeling is said to be a \emph{magic labeling}.
A magic labeling is called \emph{supermagic} if the set of edge labels consists of consecutive
positive integers. It is clear that a square $\MR(n,n;k,k)$, say $M$, produces a supermagic labeling of the graph $\Phi_M$, 
since $M$ has the same row and column sums.

Now, take a finite abelian group $\Gamma$ of order $\ell>1$, and write $\Gamma^*$ for $\Gamma \setminus \{0_\Gamma\}$.
A \emph{$\Gamma$-supermagic labeling} of a graph $G(V, E)$ with $|E| = \ell$ is a bijection from $E$ to 
$\Gamma$ such that the sum of labels of all incident edges of every vertex
$v \in V$ is equal to the same element $x \in \Gamma$, see \cite{FMMM}.
If $M$ is a $\Mo_\Gamma(m,n;s,k)$, then the function $f_M$ is a $\Gamma$-supermagic labeling of $\Phi_M$, where
the constant $x$ is $0_\Gamma$.

A graph $G$ is said to be \emph{zero-sum $\Gamma$-magic} if there exists a labeling of the edges of $G$ with 
elements of $\Gamma^*$ such that, for each vertex $v$, the sum of the labels of the edges incident with $v$ is equal to $0_\Gamma$, see \cite{ARZ}.
If $M$ is a $\Mo_{\Gamma^*}(m,n;s,k)$, then  $\Phi_M$ is a zero-sum $\Gamma$-magic graph.
Note that in this case the labeling $f_M$ is a bijection.

When the zero-sum magic array $\Mo_\Omega(m,n;s,k)$ is actually a Heffter array, there are further applications to cyclic cycle decompositions (see \cite{A});
when it is a tight $\Gamma$-magic rectangle, there are applications to cryptography, scheduling and statistical design of experiments 
(see \cite{CH} and
the references therein).
\smallskip

Finally, we briefly describe our main achievements. The first one, proved in Section \ref{rec}, extends Froncek's result
about tight magic rectangle sets, providing necessary and sufficient conditions
for the existence of an $\MRS(m,n;s,k;c)$.

\begin{thm}\label{main2}
Let  $m,n,s,k,c$ be five positive integers such that $2\leq s\leq n$, $2\leq k \leq m$ and $ms=nk$.
An $\MRS(m,n;s,k;c)$ exists if and only if either $nkc$ is odd, or $s$ and $k$ are both even and $sk>4$.
\end{thm}

Next, keeping in mind the connection with $\Gamma$-supermagic labelings and zero-sum $\Gamma$-magic graphs,
we will focus our attention on zero-sum magic partially filled arrays (sets) where $\Omega$ 
is a finite abelian group $\Gamma$, or the set $\Gamma^*$ of its nonzero elements.
In particular, we will prove the following result.

\begin{thm}\label{main}
Let  $m,n,s,k,c$ be five positive integers such that $2\leq s\leq n$, $2\leq k \leq m$ and $ms=nk$.
Set $d=\gcd(s,k)$.
\begin{itemize}
\item[(1)] A $\MoS_{\Z_{nkc}}(m,n;s,k;c)$ exists if and only if $nkc$ is odd.
\item[(2)] If  $d\equiv 0 \pmod 4$, then there exists a $\MoS_{\Gamma}(m,n;s,k;c)$ for every abelian group $\Gamma \in \G$ of order $nkc$.
\item[(3)] If $nk$ is odd and $\gcd( n, d - 1 )=1$,  
then there exists a $\Mo_{\Z_{d}\oplus \Z_{nk/d}}(m,n;s,k)$.
\item[(4)] A  $\MoS_{\Z_{2nc+1}^*}(2,n; n,2; c)$ exists if and only if $n\geq 3$.
\item[(5)] A $\Mo_{\Z_{mn+1}^\ast}(m,n;n,m)$ exists if and only if $mn>5$ is even.
\item[(6)] Let $\Gamma$ be an abelian group of order $2n+1$. There exists a tight $\Mo_{\Gamma^*}(2,n;n,2)$ if and only if 
$\Gamma \not\in \{\Z_5,\Z_3\oplus \Z_3\}$.
\end{itemize} 
\end{thm}

\section{Notation, examples and preliminary results}

Given two integers $a\leq b$, we denote by $[a,b]$ the set consisting of the integers $a,a+1,\ldots,b$.
If $a>b$, then $[a,b]$ is empty. We denote by $(i,j)$ the cell in the $i$-th row and $j$-th column 
of a partially filled array $A$, while $\E(A)$ denotes the \emph{list} of the entries of the filled cells of $A$.
We also write $\E(i,j)$ to indicate the entry of the cell $(i,j)$ of $A$.
Given a sequence $S=(B_1,B_2,\ldots,B_r)$ of partially filled arrays, we set $\E(S) = \cup_i \E(B_i)$.

We recall that a finite nontrivial abelian group can be written as  a direct sum 
$$ \Z_{n_1}\oplus \Z_{n_2}\oplus\ldots\oplus \Z_{n_\ell} $$
of cyclic groups $\Z_{n_i}$ of order $n_i>1$. In particular, we can take these integers $n_i$ in such a way that $n_i$ divides $n_{i+1}$
for all $i\in [1,\ell-1]$.
The elements of a cyclic group $(\Z_n,+)$ of order $n$ will be denoted by $[x]_n$. In other words, 
$[x]_n$ is the image of $x \in \Z$ by the canonical projection $\pi: \Z \to \Z_n$.
More in general, given a direct sum $\Gamma_1\oplus \Gamma_2\oplus \ldots \oplus \Gamma_\ell$ of abelian groups, its elements
will be denoted by $(x_1,x_2,\ldots,x_\ell)$, where $x_i \in \Gamma_i$  for all $i=1,2,\ldots,\ell$.

In the following, it will be convenient to denote an $\M_\Omega(n,n;k,k)$, an $\MS_\Omega(n,n;k,k;c)$, a $\Mo_\Omega(n,n;k,k)$ and 
a $\MoS_\Omega(n,n;k,k;c)$, respectively, by
$\M_\Omega(n;k)$, $\MS_\Omega(n;k;c)$,   $\Mo_\Omega(n;k)$ and  $\MoS_\Omega(n;k;c)$.
Furthermore, we denote a tight $\Mo_\Omega(m,n;n,m)$ by $\Mo_\Omega(m,n)$, 
and a  tight $\MoS_\Omega(m,n;n,m;c)$ by $\MoS_\Omega(m,n;c)$.

Here, some examples. The arrays
$$\begin{array}{|c|c|} \hline
[0]_4 & [1]_4 \\\hline
[3]_4 & [2]_4 \\\hline
\end{array} \equad 
\begin{array}{|c|c|} \hline
([0]_2,[0]_2) & ([0]_2,[1]_2) \\\hline
([1]_2,[0]_2) & ([1]_2,[1]_2) \\\hline
\end{array}$$
are two magic rectangles: on the left-hand side we have an $\M_{\Z_4}(2;2)$; 
on the right-hand side, we have an $\M_{\Z_2\oplus \Z_2}(2;2)$.
The arrays 
$$\begin{array}{|c|c|c|c|}\hline
([0]_2,[0]_6) & ([1]_2,[0]_6) & ([0]_2,[5]_6) & ([1]_2,[1]_6) \\\hline
([1]_2,[2]_6) & ([1]_2,[5]_6) & ([0]_2,[3]_6) & ([0]_2,[2]_6) \\\hline
([1]_2,[4]_6) & ([0]_2,[1]_6) & ([0]_2,[4]_6) & ([1]_2,[3]_6)\\\hline
\end{array}$$
and
$$\begin{array}{|c|c|c|c|}\hline
([0]_2,[0]_6) & ([0]_2,[2]_6) & ([0]_2,[4]_6) &         \\\hline
              & ([1]_2,[0]_6) & ([1]_2,[3]_6) & ([0]_2,[3]_6) \\\hline
([0]_2,[5]_6) &               & ([1]_2,[5]_6) & ([1]_2,[2]_6) \\\hline
([0]_2,[1]_6) & ([1]_2,[4]_6) &         & ([1]_2,[1]_6) \\\hline
\end{array}$$
are, respectively, a  $\Mo_\Gamma(3,4)$  and a $\Mo_\Gamma(4;3)$, where $\Gamma=\Z_2\oplus \Z_6$.

\begin{ex}
For any odd $n\geq 3$, the sum of the integers $1,2,\ldots,n$ is a multiple of $n$, being equal to $n \cdot \frac{n+1}{2}$. Then,
the array $A=(a_{i,j})$, defined by $a_{i,j}=([i]_n, [j]_n)$ for all $i,j\in  [1,n]$,
is a $\Mo_{\Z_n\oplus \Z_n}(n,n)$.
\end{ex}

We recall that, given an abelian group $\Gamma$, we set $\Gamma^\ast=\Gamma \setminus \{0_\Gamma\}$. 
Furthermore, $\G$ denotes the set of all finite abelian groups that either have odd order or contain more than one involution.
Since the sum of all the elements of  $\Gamma$ is equal to the sum of its involutions, it is easy to see that
$$\sum_{g \in \Gamma} g = \sum_{\substack{x\in \Gamma,\\ 2x=0}} x = \left\{\begin{array}{ll}
 \iota & \textrm{if } \Gamma \textrm{ has a unique involution } \iota,\\
 0 & \textrm{otherwise}.
 \end{array}\right.$$
It follows that if a $\MoS_\Omega(m,n; s,k;c)$ exists for some $\Omega\in\{\Gamma, \Gamma^\ast\}$, 
then $\Gamma\in \G$; in particular, $\Gamma$ cannot be a cyclic group of even order.
Moreover, if there exists a $\MoS_\Gamma(m,n;s,k;c)$, then either $|\Gamma|=nkc$ is odd or $nkc\equiv 0\pmod 4$.

The arrays 
$$\begin{array}{|c|c|c|c|}\hline
[1]_{13} & [2]_{13} & [10]_{13} & \\\hline
         & [3]_{13} & [4]_{13}  & [6]_{13} \\\hline
[5]_{13} &          & [12]_{13} & [9]_{13} \\\hline
[7]_{13} & [8]_{13} &           & [11]_{13} \\\hline
  \end{array} \equad 
\begin{array}{|c|c|c|c|c|c|}\hline
[1]_{13}  &    & [3]_{13}  &   & [9]_{13} &  \\ \hline
   & [2]_{13}  &    & [5]_{13} &   & [6]_{13} \\\hline 
[12]_{13} &    & [10]_{13} &   & [4]_{13} &  \\\hline
   & [11]_{13} &    & [8]_{13} &   & [7]_{13} \\\hline
\end{array}$$
are, respectively, a $\Mo_{\Z_{13}^\ast}(4;3)$ and a $\Mo_{\Z_{13}^\ast}(4,6; 3,2)$.

Let $A=(a_{i,j})$ be a partially filled square array of size $n$. We say that the
element $a_{i,j}$ belongs to the diagonal $D_{r}$ if $j-i\equiv r \pmod{n}$.
We say that $A$ is $\ell$-diagonal if the nonempty cells of $A$ are exactly those of $\ell$ consecutive diagonals.
In particular, if $A$ is an $\M_\Omega(n;k)$, then we say that $A$  is diagonal if
it is $k$-diagonal.
We also say that an $\MS_\Omega(n;k;c)$ is diagonal if every member of this set is a $k$-diagonal partially filled array
(similarly for a zero-sum magic array).
For instance, this is a diagonal  $\Mo_{\Z_{21}^\ast}(5;4)$:
$$\begin{array}{|c|c|c|c|c|}\hline	
 [1]_{21} & [19]_{21} &           &  [6]_{21} & [16]_{21} \\\hline
[20]_{21} &  [5]_{21} & [15]_{21} &           &  [2]_{21} \\\hline
[18]_{21} &  [4]_{21} &  [9]_{21} & [11]_{21} &           \\\hline
          & [14]_{21} &  [8]_{21} & [13]_{21} &  [7]_{21} \\\hline
[3]_{21}  &           & [10]_{21} & [12]_{21} & [17]_{21} \\\hline
\end{array}$$

The following theorem shows how it is possible to construct rectangular magic partially filled arrays starting from  diagonal square 
ones.
This result was actually proven in \cite{MP3} for Heffter arrays, but the proof can be easily adapted to the more general context of
magic partially filled array sets.

\begin{thm}\label{sq->rt}
Let $m,n,s,k,c$ be five positive integers such that $2\leq s \leq n$, $2\leq k\leq m$ and $ms=nk$.
Let $\Omega$ be a subset of size $nkc$ of an abelian group $\Gamma$.
Set $d=\gcd(s,k)$.
If  there exists a diagonal $\MS_\Omega\left( \frac{nk}{d}; d; c \right)$,
then there exists an  $\MS_\Omega ( m,n; s,k;c)$.
In particular, if  there exists a diagonal $\MoS_\Omega\left( \frac{nk}{d}; d;c \right)$,
then there exists a  $\MoS_\Omega ( m,n; s,k;c)$.
\end{thm}

Sets of zero-sum magic partially filled arrays can be constructed using a sort of Kronecker product.
For $i=1,2$, let $A_i$ be a partially filled array of size $a_i\times b_i$ whose elements belong to an abelian group $\Gamma_i$.
We define the ``Kronecker product'' between $A_1$ and $A_2$ as the partially filled array $A_1\otimes A_2$ of size $(a_1a_2)\times (b_1b_2)$ obtained as follows.
Set $A_1=(u_{i,j})$ and $A_2=(v_{x,y})$, and take $\ell,h$ in such a way that $\ell \in [1,a_1a_2]$ and $h\in [1,b_1b_2]$.
Write $\ell=q_1 a_2+ r_1$ and $h=q_2 b_2 + r_2$ with $1\leq r_1 \leq a_2$ and $1\leq r_2 \leq b_2$.
Then, the cell $(\ell, h)$ of $A_1\otimes A_2$ is nonempty if and only if the cell $(q_1+1,q_2+1)$ of $A_1$ and the cell $(r_1,r_2)$ of $A_2$ are nonempty. In such case,
$\E(\ell,h)=(u_{q_1+1,q_2+1}, v_{r_1,r_2}) \in\Gamma_1\oplus \Gamma_2$.

For instance, take the following partially filled arrays:
$$A_1=\begin{array}{|c|c|c|}\hline
0 &  & 1 \\\hline
2 &  3 & \\\hline
  \end{array}, \quad
  A_2=\begin{array}{|c|c|c|c|}\hline
4 & 5 &  &  \\\hline
  &  & 6 & 7\\\hline
8 &   &   &  9 \\\hline
  \end{array}.$$
Then, the product $A_1\otimes A_2$ is
$$\begin{array}{|c|c|c|c|c|c|c|c|c|c|c|c|}\hline
(0,4) &  (0,5) &       &       &       &       &       &       & (1,4) & (1,5) &       &       \\\hline 
      &        & (0,6) & (0,7) &       &       &       &       &       &       & (1,6) & (1,7) \\\hline
(0,8) &        &       & (0,9) &       &       &       &       & (1,8) &       &       & (1,9) \\\hline
(2,4) &  (2,5) &       &       & (3,4) & (3,5) &       &       &       &       &       &       \\\hline
      &        & (2,6) & (2,7) &       &       & (3,6) & (3,7) &       &       &       &       \\\hline
(2,8) &        &       & (2,9) &       &       &       & (3,9) &       &       &       &       \\\hline
 \end{array}.$$

\begin{lem}
If there exist a $\MoS_{\Omega_1}(m_1,n_1; s_1, k_1; c_1)$ and a  $\MoS_{\Omega_2}(m_2,n_2; s_2,k_2; c_2)$,
where $\Omega_1\subseteq \Gamma_1$ and $\Omega_2\subseteq \Gamma_2$, then there exists a $\MoS_{\Omega_1\times \Omega_2}(m_1m_2, n_1n_2; s_1s_2, k_1k_2; c_1c_2)$,
where $\Omega_1\times \Omega_2\subseteq \Gamma_1\oplus \Gamma_2$.
\end{lem}

\begin{proof}
Let $A_i$ be an array of the set $\MoS_{\Omega_1}(m_1,n_1; s_1, k_1; c_1)$ and $B_j$ be an array of the set $\MoS_{\Omega_2}(m_2,n_2; s_2,k_2; c_2)$.
Then $C=A_i\otimes B_j$ is a partially filled array of size $(m_1m_2)\times (n_1n_2)$ such that: 
every row of $C$ contains $s_1s_2$ elements  and each column of 
$C$ contains $k_1k_2$ elements; the elements of each row of $C$ sum to $(s_2\cdot 0_{\Gamma_1}, s_1\cdot 0_{\Gamma_2} )=0_\Gamma$; the elements of each column of $C$ sum to $(k_2\cdot 0_{\Gamma_1}, k_1\cdot 0_{\Gamma_2} )=0_\Gamma$; $\E(C)=\{(x,y)\mid x \in \E(A_i), y \in \E(B_j)\}$.
We conclude that the set $\{A_i\otimes B_j\mid i \in [1,c_1], j \in [1,c_2] \}$ is a $\MoS_{\Omega_1\times \Omega_2}(m_1m_2, n_1n_2; s_1s_2, k_1k_2; c_1c_2)$.
\end{proof}

\section{Magic rectangle sets with empty cells}\label{rec}

In this section we solve the existence problem of a magic rectangle set with empty cells.
Our starting point is Froncek's result about (tight) magic rectangle sets.
As we recalled in Section \ref{intro}, Froncek studied these objects, proving the following.

\begin{thm}\cite[Theorem 3.2]{F3}\label{Fron}
Let $m,n,c$ be positive integers such that $m,n\geq 2$.
A magic rectangle set $\MRS(m, n; n,m; c)$ exists if and only if either $mnc$ is odd, or $m,n$ are both even and $mn>4$ ($c$ arbitrary).
\end{thm}

\begin{rem}\label{nec_rec}
Suppose that an $\MRS(m,n; s,k; c)$ exists. 
Then, for every array of this set, the elements of each row and each column sum, respectively, to $\frac{s(nkc+1)}{2}$ 
and to $\frac{k(nkc+1)}{2}$. So, if $nkc$ is even, then $s$ and $k$ must be both even.
\end{rem}

First, we consider the case when $k$ (or $s$) is equal to $2$.

\begin{prop}\label{rec2}
Let  $m,n,s,c$ be four positive integers such that $2\leq s\leq n$  and $ms=2n$.
An $\MRS(m,n; s,2; c)$  exists if and only if $s\geq 4$ is even. 
\end{prop}

\begin{proof}
Set $N=2nc+1$.
Suppose that an $\MRS(m,n; s,2; c)$ exists.
By Remark \ref{nec_rec},  $s$ must be even.
Furthermore, it is easy to see that an $\MRS(n,n; 2,2; c)$ does not exist. We conclude that $s$ must be an even
integer greater than $2$.

Now, suppose that $s\geq 4$ is even and write $s=2\bar s$. Hence, we have $n=m\bar s$.
Our construction of an  $\MRS(m, m \bar s; 2 \bar s,2; c)$ depends on the parity of $\bar s$, and uses some basic blocks.
Given an integer $x\geq 0$, we construct two $2$-diagonal $m\times m$ partially filled arrays $U_x, V_x$:
$$\begin{array}{rcl}
U_x & =& \begin{array}{|c|c|c|c|c|}\hline
x+1     & N-(x+2) &          &          &  \\\hline
        & x+2     & N- (x+3) &          &  \\\hline
        &         & \ddots   &  \ddots  &  \\\hline
        &         &          &  x+(m-1) &  N-(x+m)\\\hline
N-(x+1) &         &          &          &  x+m \\\hline
 \end{array}, \\ \\[-8pt]
V_x & =& \begin{array}{|c|c|c|c|c|}\hline
x+m     & N-(x+m-1) &          &          &  \\\hline
        & x+(m-1) & N- (x+m-2) &          &  \\\hline
        &         & \ddots   &  \ddots  &  \\\hline
        &         &          &  x+2 &  N-(x+1)\\\hline
N-(x+m) &         &          &          &  x+1 \\\hline
 \end{array}.
 \end{array}$$
In both cases, the elements of each column sum to $N$; the row sums of $U_x$ are $(N-1, \ldots, N-1, N+m-1)$, while
the row sums of $V_x$ are $(N+1,\ldots, N+1, N-m+1)$.
Furthermore, $\E(U_x)=\E(V_x)=[x+1, x+m]\cup [N-(x+m), N-(x+1)]$.
Now, let $A_x$ be the partially filled array obtained by the juxtaposition of $U_x$ and $V_{x+m}$.
Then, $A_x$ is an $\M_\Omega(m,2m; 4,2)$, where $\Omega=\E(A_x)=[x+1,x+2m]\cup [N-(x+2m), N-(x+1)]$, 
the elements of every column  sum to $N$, and the elements of every row sum to $2N$.

Assume $\bar s= 2t$ with $t\geq 1$. For every $\ell\geq 0$, let $R_\ell$ be the  $m\times n$ partially filled array obtained by the
juxtaposition of $A_{2m (t\ell)}, A_{2m (t\ell+1)}, \ldots,A_{2m(t\ell+t-1)}$.
By construction, every column of $R_\ell$ contains $2$ filled cells, and every row contains $4 t=s$ filled cells.
The elements of each column sum to $N$, while the elements of each row sum to $2 N t = N\bar s$. 
Furthermore, 
$$\E(R_\ell)=[m\bar s \ell+1, m\bar s (\ell+1)] \cup [N-m\bar s(\ell+1), N-(m\bar s \ell+1)].$$
Our  $\MRS(m,n; s,2; c)$ consists of $R_0, R_1, \ldots, R_{c-1}$.
In fact,
$$\bigcup_{\ell=0}^{c-1} \E(R_\ell)=[1, m\bar s  c] \cup [N- m\bar s c, N-1]=[1,nc]\cup [nc+1,2nc]=[1,2nc].$$

Next, assume $\bar s= 2t+1$ with $t\geq 1$. Given an integer $x\geq 0$, we construct a $2$-diagonal $m\times m$ partially filled array
$W_x$ as follows:
$$W_x = \begin{array}{|c|c|c|c|c|}\hline
x+(2m-1)   & N-(x+2m-3) &             &          &  \\\hline
           & x+(2m-3)   & N- (x+2m-5) &          &  \\\hline
           &            & \ddots      &  \ddots  &  \\\hline
           &            &             &  x+3     &  N-(x+1)\\\hline
N-(x+2m-1) &        &             &          &  x+1 \\\hline
 \end{array}.$$
The elements of each column of $W_x$ sum to $N$, while the row sums are $(N+2,\ldots,N+2, N+2-2m)$.
Furthermore, $\E(W_x)=\{x+1, x+3, \ldots, x+(2m-1)\}\cup \{N-(x+2m-1), N-(x+2m-3), \ldots, N-(x+1)\}$.

Let $Y_x,\widetilde Y_x$  be the partially filled arrays obtained by the juxtaposition of $U_x, U_{x+m}, W_{x+2m}$
and of  $W_{x+1}, U_{x+2m}, U_{x+3m}$, respectively.
Then, $Y_x$ and $\widetilde Y_x$ are two $\M_\Omega(m,3m; 6,2)$, where
$$\begin{array}{rcl}
\Omega=\E(Y_x) & =& [x+1, x+2m]\cup [N-(x+2m), N-(x+1)] \\
&& \cup \{x+2m+1, x+2m+3, \ldots, x+4m-1\}\\
&& \cup \{N-(x+4m-1), N-(x+4m-3), \ldots, N-(x+2m+1)\},\\
\Omega=\E(\widetilde Y_x) & =& \{x+2, x+4, \ldots, x+2m\}\cup 
\{N-(x+2m), N-(x+2m-2), \ldots,\\
&& N-(x+2)\} \cup [x+2m+1, x+4m] \cup [N-(x+4m), N-(x+2m+1)].   
  \end{array}$$
In both cases,  the elements of each row and each column  sum, respectively, to $3N$ and to $N$.

For every $\ell\geq 0$, let $R_{2\ell}, R_{2\ell+1}$ be the  $m\times n$ partially filled arrays obtained by the
juxtaposition, respectively, of 
$$A_{2m(2t+1)\ell}, A_{2m ((2t+1)\ell+1)}, \ldots,A_{2m((2t+1)\ell+t-2)}, Y_{2m((2t+1)\ell+t-1)},$$
and of 
$$\widetilde Y_{2m((2t+1)\ell+t)}, A_{2m((2t+1)\ell+t+2)}, A_{2m((2t+1)\ell+t+3)},\ldots,A_{2m((2t+1)\ell+2t)}.$$
In both cases,  every column contains $2$ filled cells and every row contains $4 (t-1)+6= s$ filled cells;
the elements of each column sum to $N$, while the elements of each row sum to $2N (t-1)+3N = N\bar s$. 
Furthermore, 

$$\begin{array}{rcl}
\E(R_{2\ell}) & =& [ms \ell + 1,  ms\ell + 2mt]\cup [N-(ms\ell + 2mt), N-(ms \ell + 1)] \\
& & \cup \{ms\ell + 2mt+1, ms\ell + 2mt+3, \ldots, ms\ell + 2mt+ 2m-1 \}\\
&& \cup \{N-( ms\ell + 2mt+ 2m-1), N-( ms\ell + 2mt+ 2m-3),\ldots, \\
&& N-(ms\ell + 2mt+1)\},\\
\E(R_{2\ell+1}) & =&  \{ms\ell + 2mt+2, ms\ell + 2mt+4, \ldots, ms\ell + 2mt+ 2m   \}\\
&&\cup \{N-( ms\ell + 2mt+ 2m), N-( ms\ell + 2mt+ 2m-2),\ldots, \\
&& N-(ms\ell + 2mt+2)\}\cup [ms \ell +2m(t+1) + 1, ms(\ell+1)]\\
&& \cup [N-ms(\ell+1) , N-(ms \ell +2m(t+1) + 1) ].
  \end{array}$$ 
Note that
$$\E(R_{2\ell})\cup \E(R_{2\ell+1})=[ms\ell+1, ms(\ell+1)]\cup [N-ms(\ell+1),N-(ms\ell+1)].$$
Our  $\MRS(m,n; s,2; c)$ consists of $R_0, R_1, \ldots, R_{c-1}$.
In fact, if $c$ is even, then
$$\bigcup_{\ell=0}^{(c-2)/2} \left(\E(R_{2\ell})\cup \E(R_{2\ell+1})\right)=[1, m\bar s  c] \cup [N- m\bar s c, N-1]=[1,2nc].$$
If $c$ is odd, then $\E(R_{c-1})= [m\bar s (c-1)+1, m\bar s (c+1)]$ and
$$\begin{array}{rcl}
\E(R_{c-1})\cup \bigcup\limits_{\ell=0}^{(c-3)/2} \left(\E(R_{2\ell})\cup \E(R_{2\ell+1})\right) & =& 
[m\bar s (c-1)+1, m\bar s (c+1)] \cup [1, m\bar s(c-1) ]\\
&& \cup [m\bar s (c+1)+1, 2m\bar s c]\\
& =& [1,2nc].   
  \end{array}$$
  This concludes our proof.
\end{proof}

As recalled in Section \ref{intro}, a signed magic array $\SMA(m,n;s,k)$ is a $\Mo_\Omega(m,n; s,k)$, where
$\Omega=\left[-\frac{nk-1}{2}, +\frac{nk-1}{2}\right]$ if $nk$ is odd or 
$\Omega=\left[-\frac{nk}{2}, -1\right]\cup  \left[1, \frac{nk}{2}\right]$ if $nk$ is even.

\begin{prop}\label{magic_rect}
Let  $m,n,s,k,c$ be five positive integers such that $3\leq s\leq n$, $3\leq k \leq m$ and $ms=nk$.
There exists an $\MRS(m,n;s,k;c)$ if and only if either $nkc$ is odd, or $s$ and $k$ are both even.
\end{prop}

\begin{proof}
Set $d=\gcd(s,k)$ and suppose that the product $nkc$ is odd.
Assume $d\neq 1$, hence $d\geq 3$. By \cite[Corollary 7]{KL}, there exists a diagonal $\MRS\left(\frac{nk}{d}; d ; c\right)$:
hence, the existence of an $\MRS(m,n;s,k;c)$  follows from  Theorem \ref{sq->rt}.
Assume $d=1$. Since $s$ and $k$ are coprime, we can write $m=kr$ and $n=sr$, where $r\geq 1$ is an odd integer.
Let $A_0,\ldots,A_{rc-1}$ be the arrays of a magic rectangle set $\MRS(k,s;rc)$, whose existence is guaranteed by Theorem \ref{Fron}.
Write $A_\ell=\left(a_{i,j}^{(\ell)}\right)$, with $i \in [1,k]$, $j \in [1,s]$ and $\ell\in [0,rc-1]$.
We construct $c$ partially filled arrays of size  $m\times n$  by taking
empty arrays $R_t$, $t\in [0,c-1]$, of size $(kr)\times (sr)$ and filling them in such a way that 
the entry of the cell $(k u+i,su+j)$, $u \in [0,r-1]$, of $R_t$ is  $a_{i,j}^{(rt+u)}$.
Hence, the arrays $R_0,\ldots,R_{c-1}$ so constructed are the members of an $\MRS(m,n;s,k;c)$.
In fact, its entries are the integers of $[1, k s rc]=[1,nkc]$, 
each row of $R_t$ contains $s$ filled cells and each column contains $k$ filled cells; 
the elements of every row sum to $s\frac{ksrc+1 }{2}=s\frac{nkc+1}{2}$, while the elements of every column sum to 
$k\frac{nkc+1}{2}$.

Now, suppose that $s$ and $k$ are both even. By \cite[Proposition 5.7]{MP3} there exists a shiftable $\SMA(m,n;s,k)$, say $A$,
that is, a signed magic array where every row and every column contains an equal number
of positive and negative entries.
For every $\ell\in [1,c]$, let $R_\ell$ be the array obtained from $A$ by replacing every positive entry $x$ of $A$  with
$x+\frac{nk}{2}(2c-\ell)$ and replacing every negative entry $y$ with $y+ \frac{nk}{2}\ell+1$.
So, the elements of each row of $R_\ell$ sum to $\frac{s}{2}( \frac{nk}{2}(2c-\ell)+   \frac{nk}{2}\ell+1 )=\frac{s}{2} (nkc+1) $
and the elements of each column sum to $\frac{k}{2} (nkc+1)$.
Furthermore, $\E(R_\ell)=\left[ \frac{nk}{2}(\ell-1)+1, \frac{nk}{2}\ell\right]\cup \left[\frac{nk}{2}(2c-\ell)+1,
\frac{nk}{2}(2c+1-\ell)\right]$.
It follows that
$$\bigcup_{\ell=1}^c \E(R_\ell)=\left[1,\frac{nk}{2}c\right]\cup \left[\frac{nk}{2}c+1,\frac{nk}{2}2c\right] =[1,nkc],$$
and so $R_1,\ldots,R_c$ are the members of an $\MRS(m,n;s,k;c)$.

This proves  the existence of an $\MRS(m,n;s,k;c)$ whenever $nkc$ is odd or $n$ and $k$ are both even. 
Vice-versa, if an $\MRS(m,n;s,k;c)$ exists, either $nkc$ is odd, or $s$ and $k$ are both even, by Remark \ref{nec_rec}.
\end{proof}

\begin{proof}[Proof of Theorem \ref{main2}]
If $s,k\geq 3$, the result follows from Proposition \ref{magic_rect}.
Also, since the transpose of an $\M_\Omega(m,n;s,k)$ is an $\M_\Omega(n,m;k,s)$, we may assume $k=2$, and hence the result
follows from Proposition \ref{rec2}.
\end{proof}

\begin{cor}\label{SMAodd}
Let  $m,n,s,k$ be four integers such that $3\leq s\leq n$, $3\leq k \leq m$ and $ms=nk$.
If $nk$ is odd, then there exists an $\SMA(m,n;s,k)$.
\end{cor}

\begin{proof}
Let $A$ be an $\MR(m,n;s,k)$, whose existence follows from Theorem \ref{main2} for $c=1$.
Set $w=\frac{nk+1}{2}$, so that the elements of each row of $A$ sum to $sw$, while the elements  of each column sum to $kw$.
Replacing each entry $x\in [1,nk]$ of $A$ with $x-w$, we obtain an $\SMA(m,n;s,k)$, say $B$.
In fact, $\E(B)=[1-w, nk-w]=\left[-\frac{nk-1}{2}, +\frac{nk-1}{2}\right]$. Moreover, the elements of each row of $B$ sum to $sw - s w=0$;
similarly, the elements of each column sum to $0$.
\end{proof}

\section{Some constructions for the case $\Omega=\Gamma$}\label{Gamma}

First, we consider the case when $\Gamma$ is a cyclic group. Thanks to the results of Section \ref{rec} we obtain the following.

\begin{cor}\label{cyc}
Let  $m,n,s,k,c$ be five positive integers such that $2\leq s\leq n$, $2\leq k \leq m$ and $ms=nk$.
A $\MoS_{\Z_{nkc}}(m,n; s,k; c)$ exists if and only if $nkc$ is odd.
\end{cor}

\begin{proof}
If a  $\MoS_{\Z_{nkc}}(m,n;s,k;c)$ exists, then $nkc$ must be odd.
In this case, by Theorem \ref{main2} there exists an $\MRS(m,n; s,k; c)$,
whose members are, say, $A_1,\ldots,A_c$.
Write $w=\frac{nkc+1}{2}$ and replace each entry $x\in [1, nkc]$ of $A_\ell$ with 
$[x-w]_{nkc} \in \Z_{nkc}$, obtaining a partially filled array $B_\ell$.
Then, $\{B_1,\ldots,B_c\}$ is a $\MoS_{\Z_{nkc}}(m,n;s,k;c)$.
In fact, the elements of each row of $B_\ell$ sum to $[s w ]_{nkc}-s [w]_{nkc} =[0]_{nkc}$;
similarly, the elements of each column sum to $[0]_{nkc}$.
Furthermore, 
$$\begin{array}{rcl}
\bigcup\limits_{\ell=1}^c \E(B_\ell) & =&  \left\{ [1-w]_{nkc}, [2-w]_{nkc},\ldots, [nkc-w]_{nkc} \right\}\\
& =& \left\{ 
\left[-\frac{nk-1}{2}\right]_{nkc}, 
\left[-\frac{nk-1}{2}+1\right]_{nkc}, \ldots,
\left[ \frac{nk-1}{2}\right]_{nkc} \right\}= \Z_{nkc}.
\end{array}$$
\end{proof}

\begin{ex}
We start by taking the arrays  $A_0,A_1,A_2$  of an $\MRS(3,5;3)$:
$$A_0=\begin{array}{|c|c|c|c|c|}\hline
 1 & 35 &  9 & 31 & 39   \\\hline
24 & 28 & 20 & 27 & 16   \\\hline
44 &  6 & 40 & 11 & 14  \\\hline
\end{array},\quad
A_1=\begin{array}{|c|c|c|c|c|}\hline
 2 & 36 &  7 & 32 & 38 \\\hline
45 &  4 & 41 & 12 & 13 \\\hline
22 & 29 & 21 & 25 & 18  \\\hline
\end{array},$$
$$A_2=\begin{array}{|c|c|c|c|c|}\hline
 3 & 34 &  8 & 33 & 37 \\\hline
43 &  5 & 42 & 10 & 15 \\\hline
23 & 30 & 19 & 26 & 17 \\\hline
\end{array}.$$
Following the proofs of Proposition \ref{magic_rect} and  Corollary \ref{cyc},
we get the $\Mo_{\Z_{45}}(9,15; 5,3)$ of Figure \ref{fig1}.
\end{ex}

\begin{figure}[ht]
$$\begin{array}{|c|c|c|c|c|c|c|c|c|c|c|c|c|c|c|}\hline
23 & 12 & 31 &  8 & 16  & & & & & & & & & & \\\hline
 1 &  5 & 42 &  4 & 38  & & & & & & & & & &\\\hline
21 & 28 & 17 & 33 & 36  & & & & & & & & & &\\\hline
& & & & & 24 & 13 & 29 &  9 & 15 & & & & &  \\\hline
& & & & & 22 & 26 & 18 & 34 & 35 & & & & &  \\\hline
& & & & & 44 &  6 & 43 &  2 & 40 & & & & & \\\hline
& & & & & & & & & & 25 & 11 & 30 & 10 & 14   \\\hline
& & & & & & & & & & 20 & 27 & 19 & 32 & 37   \\\hline
& & & & & & & & & &  0 &  7 & 41 &  3 & 39  \\\hline
\end{array}.$$
\caption{A $\Mo_{\Z_{45}}(9,15; 5,3)$, where each entry $x$ must be read as $[x]_{45}$.}\label{fig1}
\end{figure}

Now, we completely solve the case when $k\equiv 0 \pmod 4$ (there is no need to assume that $\Gamma$ is noncyclic).

\begin{prop}\label{d0}
Suppose $n\geq k\geq 4$ with $k\equiv 0 \pmod 4$. Then, there exists a diagonal $\MoS_{\Gamma}(n;k;c)$ for every $\Gamma \in \G$ of order $nkc$.
\end{prop}

\begin{proof}
Since $k\equiv 0\pmod 4$ and $\Gamma \in \G$, we can write $\Gamma = \Z_{2a}\oplus \Z_{2b}\oplus \Gamma'$, where $a,b\geq 1$ and 
$|\Gamma'|\geq 1$.
Let $\Lambda=[1,a]\times [1,b] \times \Gamma'\subset \Z\oplus \Z \oplus \Gamma'$.
For any $(x,y,g) \in \Lambda$, consider the following $3\times 2$ partially filled array with elements in $\Gamma$:
$$B(x,y,g)=
\begin{array}{|c|c|}\hline
([2x]_{2a}, [2y]_{2b},g) & -([2x+1]_{2a}, [2y]_{2b},g) \\\hline
& \\\hline
-([2x]_{2a}, [2y+1]_{2b}, g) & ([2x+1]_{2a}, [2y+1]_{2b}, g)\\ \hline
\end{array}.$$

Note that the elements in the nonempty rows sum to 
$([-1]_{2a}, [0]_{2b}, 0_{\Gamma'})$ and $([1]_{2a}, [0]_{2b}, 0_{\Gamma'})$,
while the elements of the columns sum to  $([0]_{2a}, [-1]_{2b}, 0_{\Gamma'})$ and $([0]_{2a}, [1]_{2b}, 0_{\Gamma'})$.
We  use this $3\times 2$ block for constructing  partially filled  arrays whose rows and columns sum to $0_\Gamma$.
Define $\B=\{ B(x,y,g)\mid (x,y,g)\in \Lambda\}$: this is a set of cardinality $ab|\Gamma'|=\frac{|\Gamma|}{4}=\frac{nkc}{4}$.
So, write  $\B=\left(X_1,X_2,\ldots,X_{\frac{k}{4}nc}\right)$.
Taking an empty $n\times n$ array $A_1$, arrange the first $n$ blocks of $\B$ in such a way 
that the element of the cell $(1,1)$ of $X_j$
fills the cell $(j,j)$ of $A_1$ (we work modulo $n$ on row/column indices).
In this way, we fill the diagonals $D_{n-2}, D_{n-1}, D_{0}, D_{1}$.
In particular, every row has $4$ filled cells and every column has $4$ filled cells.
Looking at the rows, the elements belonging to the
diagonals $D_{0},D_{1}$ sum to $([-1]_{2a}, [0]_{2b}, 0_{\Gamma'})$, 
while the elements belonging to the diagonals $D_{n-2},D_{n-1}$ sum to $([1]_{2a}, [0]_{2b}, 0_{\Gamma'})$.
Looking at the columns, the elements belonging to the diagonals $D_{0},D_{n-2}$
sum to $([0]_{2a}, [-1]_{2b}, 0_{\Gamma'})$, while the elements belonging to the diagonals $D_{1},D_{n-1}$
sum to $([0]_{2a}, [1]_{2b}, 0_{\Gamma'})$. Then $A_1$ has row/column sums equal to $0_\Gamma$.

Applying this process $\frac{k}{4}$ times (working with $X_{n+1},X_{n+2}, \ldots,
X_{2n}$ on the diagonals $D_{2},D_{3},$ $D_{4},D_{5}$, and so on),
we obtain a partially filled array $A_1$, whose rows and columns have exactly $k$ filled cells.
Finally, we repeat this entire process $c-1$ times, obtaining a set $A_1,\ldots,A_c$ of partially filled arrays.
To prove that this set is a diagonal $\MoS_{\Gamma}(n;k;c)$ it suffices to check that $\bigcup\limits_{\ell=1}^c\E(A_\ell)=\E(\B)$
is equal to $\Gamma$.

Considering the four entries of each $X_i$, we can write $\E(\B)=S_1\cup S_2\cup S_3\cup S_4$, where
$$\begin{array}{rcl}
S_1 & = & \{ ([2x]_{2a}, [2y]_{2b}, g)\mid (x,y,g)\in \Lambda \},\\    
S_2 & = & \{- ([2x+1]_{2a}, [2y]_{2b}, g) \mid (x,y,g)\in \Lambda \},\\    
S_3 & = & \{- ([2x]_{2a}, [2y+1]_{2b}, g) \mid (x,y,g)\in \Lambda  \},\\    
S_4 & = & \{  ([2x+1]_{2a}, [2y+1]_{2b}, g)\mid (x,y,g)\in \Lambda \}.
  \end{array}$$
Clearly, the sets $S_1,S_2,S_3,S_4$ are pairwise disjoint. 
Fixed $x,y\in \Z$, we have $-[2x+1]_{2a}=[2x'+1]_{2a}$ and $-[2y]_{2b}=[2y']_{2b}$ for some $x' \in [1,a]$ and some $y'\in [1,b]$.
It follows that
$$\begin{array}{rcl}
S_2& = & \{([2x+1]_{2a}, [2y]_{2b}, g) \mid (x,y,g)\in \Lambda \},\\
S_3 & = & \{([2x]_{2a}, [2y+1]_{2b}, g) \mid (x,y,g)\in \Lambda \}.
\end{array}$$
Then $\E(\B)=\Z_{2a}\oplus \Z_{2b}\oplus \Gamma'=\Gamma$.
 \end{proof}
 
Following the proof of the previous proposition, we construct a $\MoS_{\Z_6\oplus \Z_2\oplus \Z_4}(6;4;2)$,
where each entry $xyz$ must be read as $([x]_6,[y]_2,[z]_4)$:
 $$\begin{array}{|c|c|c|c|c|c|}\hline
000 & 500 &     &     & 410 & 310 \\\hline
311 & 001 & 503 &     &     & 413 \\\hline
010 & 110 & 002 & 502 &     &      \\\hline
    & 013 & 111 & 003 & 501 &      \\\hline
    &     & 012 & 112 & 200 & 300  \\\hline
303 &     &     & 011 & 113 & 201  \\\hline
   \end{array},\quad
\begin{array}{|c|c|c|c|c|c|}\hline
202 & 302 &     &     & 212 & 512 \\\hline
513 & 203 & 301 &     &     & 211 \\\hline
412 & 312 & 400 & 100 &     &      \\\hline
    & 411 & 313 & 401 & 103 &     \\\hline
    &     & 210 & 510 & 402 & 102 \\\hline
101 &     &     & 213 & 511 & 403 \\\hline
   \end{array}.$$

Now, we provide a construction when $k$ is odd.

\begin{prop}\label{diag_k}
Suppose that $n\geq k\geq 3$ are odd integers such that  $\gcd( n, k-1 )=1$. Then there exists a diagonal $\Mo_{\Z_k\oplus \Z_n}(n;k)$.
\end{prop}

 \begin{proof}
Write $n=2a+1$ and $k=2b+1$. Take an empty $n\times n$ array and fill the diagonals $D_{n-b},\ldots, D_{n-1},D_0,D_1,\ldots,D_b$ as follows.
Note that we give the elements for each diagonal, starting with those belonging to the first row.
For $\ell\in [1,b]$, 
$$\begin{array}{rl}
D_\ell :&  ([2\ell]_k, [1]_n), \;([2\ell]_k, [2]_n),\; \ldots,([2\ell]_k, [n]_n);\\ \\[-8pt]
D_{n-\ell} : &  ([2\ell-1]_k, [1]_n),\; ([2\ell-1]_k, [2]_n),\; \ldots,([2\ell-1]_k, [n]_n);\\ \\[-8pt]
D_0 : &  ([0]_k, [1-k ]_n),\; ([0]_k, [2(1-k) ]_n)\; \ldots,([0]_k, [n(1-k)]_n).
\end{array}$$
Call $A$ the partially filled array so obtained.
By construction, we fill $k$ cells in each row and each column of $A$. Also, we have 
$$\bigcup_{\ell=1}^b \left(\E(D_\ell)\cup \E(D_{n-\ell})\right) =\left\{(x,y)\mid x \in \Z_k^*, \; y \in \Z_n\right\}$$
and $\E(D_0)=\{([0]_k,y)\mid y \in \Z_n\}$, whence $\E(A)=\Z_k\oplus \Z_n$.
For every $i \in [1,n]$, the $i$-th row of $A$ contains the element $([0]_k, [i(1-k)]_n)$ and the elements
$([x]_k, [i]_n)$ with $x \in [1,k-1]$. The sum of these elements is
$$([0]_k, [i(1-k)]_n) + \sum_{x=1}^{k-1} ([x]_k, [i]_n) = ([0]_k, [i(1-k)+(k-1)i]_n)=([0]_k, [0]_n).$$
The $i$-th column of $A$ contains the element $([0]_k, [i(1-k)]_n)$ and the elements
$([2\ell-1]_k,[\ell+i]_n )$, $([2\ell]_k,[n-\ell+i]_n)$  with $\ell \in [1,b]$.
The sum of these elements is
$$\begin{array}{c}
([0]_k, [i(1-k)]_n)+ \sum\limits_{\ell=1}^b \left(([2\ell-1]_k,[\ell+i]_n ) + ([2\ell]_k,[i-\ell]_n)\right)=\\ \\[-8pt]
([0]_k, [i(1-k)]_n)+ ([0]_k, [2ib]_n)= ([0]_k, [i(1-k) +i(k-1)]_n)=([0]_k, [0]_n).
\end{array}$$
This proves that $A$ is a  $\Mo_{\Z_k\oplus \Z_n}(n;k)$.
\end{proof}

Here, a diagonal $\Mo_{\Z_5\oplus\Z_7}(7;5)$ obtained following the proof of the previous proposition:
$$\begin{array}{|c|c|c|c|c|c|c|} \hline
([0]_5,[3]_7) & ([2]_5,[1]_7) & ([4]_5,[1]_7) &               &               & ([3]_5,[1]_7) & ([1]_5,[1]_7) \\\hline
([1]_5,[2]_7) & ([0]_5,[6]_7) & ([2]_5,[2]_7) & ([4]_5,[2]_7) &               &               & ([3]_5,[2]_7) \\\hline
([3]_5,[3]_7) & ([1]_5,[3]_7) & ([0]_5,[2]_7) & ([2]_5,[3]_7) & ([4]_5,[3]_7) &               &               \\\hline
              & ([3]_5,[4]_7) & ([1]_5,[4]_7) & ([0]_5,[5]_7) & ([2]_5,[4]_7) & ([4]_5,[4]_7) &               \\\hline
              &               & ([3]_5,[5]_7) & ([1]_5,[5]_7) & ([0]_5,[1]_7) & ([2]_5,[5]_7) & ([4]_5,[5]_7) \\\hline
([4]_5,[6]_7) &               &               & ([3]_5,[6]_7) & ([1]_5,[6]_7) & ([0]_5,[4]_7) & ([2]_5,[6]_7) \\\hline
([2]_5,[0]_7) & ([4]_5,[0]_7) &               &               & ([3]_5,[0]_7) & ([1]_5,[0]_7) & ([0]_5,[0]_7) \\\hline
\end{array}.$$

\section{Some constructions for the case $\Omega=\Gamma^*$}

Also in this case, we firstly consider the cyclic case. 
Thanks to Theorem \ref{main2} we obtain the following result.

\begin{lem}\label{rect->MPFS}
Let  $m,n,s,k,c$ be five positive integers such that $2\leq s\leq n$, $2\leq k \leq m$ and $ms=nk$.
Suppose that $s$ and $k$ are both even and such that $(s,k)\neq (2,2)$.
Then, there exists a $\MoS_{\Z_{nkc+1}^*}(m,n;s,k; c)$. 
\end{lem}

\begin{proof}
By Theorem \ref{main2}, there exists an  $\MRS(m,n;s,k; c)$, say $\B=\{A_1,A_2,\ldots,A_c\}$.
Call $\widetilde \B=\{B_1,B_2,\ldots,B_c\}$ the set obtained by replacing each entry $x\in \Z$ of $A_i$ 
with $[x]\in \Z_{nkc +1}$. Then, $\E(\widetilde \B)=\{[x]_{nkc+1}\mid x \in [1,nkc]\}=\Z_{nkc+1}^*$.
Since $s$ is even and the elements of each row of $A_i$ sum to $\frac{s}{2} (nkc+1)$, the elements
of each row of $B_i$ sum to $[0]_{nkc+1}$. Similarly for the columns.
We conclude that $\widetilde \B$ is a $\MoS_{\Z_{nkc+1}^*}(m,n;s,k; c)$. 
\end{proof}

We also make use of the known results about signed magic arrays.
We recall that in \cite{KSW} it was proved that an $\SMA(2,n;n,2)$ exists if and only if $n\equiv 0,3 \pmod 4$, while an $\SMA(m,n;n,m)$ exists for all $m,n>2$.

\begin{thm}\cite{MP3}\label{SMA}
Let  $m,n,s,k$ be four integers such that $3\leq s\leq n$, $3\leq k \leq m$ and $ms=nk$.
There  exists an $\SMA(m,n;s,k)$ whenever $\gcd(s,k)\geq 2$, or $s\equiv 0 \pmod 4$,  or $k\equiv 0 \pmod 4$. 
Furthermore, there exists a diagonal $\SMA(n,n;k,k)$ for any $n\geq k \geq 3$.
\end{thm}

Replacing the entry $x\in \Z$ of an $\SMA(m,n;s,k)$, with the element $[x]_{nk+1}$ of $\Z_{nk+1}$,
we get the following.

\begin{lem}\label{SMA->MPF}
Suppose there exists an $\SMA(m,n;s,k)$, where $nk$ is even.
Then there exists a $\Mo_{\Z_{nk+1}^\ast}(m,n;s,k)$.
\end{lem}

From Theorem \ref{SMA} and Lemma \ref{SMA->MPF} we obtain the following result.

\begin{cor}\label{corSMA->MPF}
Let  $m,n,s,k$ be four integers such that $2\leq s\leq n$,  $2\leq k \leq m$   and  $ms=nk$.
If $nk$ is even, then there exists a $\Mo_{\Z_{nk+1}^\ast}(m,n;s,k)$ in each of the following cases:
\begin{itemize}
 \item[(1)] $s=n\geq 3$ and $k=m\geq 3$, or $m=n$ and $s=k\geq 3$;
 \item[(2)] $k=m=2$ and $s=n \equiv 0,3 \pmod 4$, or $s=n=2$ and $k=m \equiv 0,3 \pmod 4$;
 \item[(3)] $s\equiv 0\pmod 4$ or $k \equiv 0 \pmod 4$;
 \item[(4)] $\gcd(s,k)\geq 2$.
\end{itemize}
\end{cor}

\begin{prop}\label{ciclico}
A tight $\MoS_{\Z_{2nc+1}^*}(2,n; c)$ exists if and only if $n\geq 3$.
\end{prop}

\begin{proof}
As already observed, there is no $\MoS_{\Z_{2nc+1}^*}(2,n;c)$ when $n=1,2$. 
If $n\geq 4$ is even, we apply Lemma \ref{rect->MPFS} taking $m=k=2$ and $n=s$. So, we may assume that $n\geq 3$ is odd.

We first consider the case $c\geq 3$. Let $H$ be a Heffter array $\H(c,n;n,c)$ (see Definition \ref{def:H}), whose existence was proved in \cite{ABD}.
So, there exists a subset $\Omega$ of $\Z_{2nc+1}$ such that $\Omega\cup -\Omega$ is a partition of $\Z_{2nc+1}^*$.
Let $R_1,R_2,\ldots,R_c$ be the rows of $H$: by definition, the elements of each $R_i$ sum to $[0]_{2nc+1}$.
For every $i\in [1,c]$ we construct the $2\times n$ array $A_i$, by taking
$$A_i=\begin{array}{|c|}\hline R_i \\ \hline -R_i \\\hline\end{array}.$$
By construction, the elements of each column of $A_i$ sum to $[0]_{2nc+1}$. Also, taking $\B=\{A_1,A_2,\ldots,A_c\}$, we have
$\E(\B)=\Omega \cup -\Omega = \Z_{2nc+1}^*$, showing that $\B$ is a $\MoS_{\Z_{2nc+1}^*}(2,n;c)$.

Now, we deal with the case $c=1$. If $n\equiv 3\pmod{4}$, then the result follows from Corollary \ref{corSMA->MPF}.
So, suppose $n\equiv 1 \pmod 4$, and write $n=4\ell+5$ with  $\ell \geq 0$.
Take the following subsets of $[1,n]$:
$$\Lambda_+ = \{\ell+1\} \cup [\ell+3, 3\ell+5] \equad 
\Lambda_-  = [1,\ell] \cup \{\ell+2\} \cup [ 3\ell+6,  4\ell+5].$$
Then $\Lambda_+\cup \Lambda_- = [1, n]$ and
$$\begin{array}{rcl}
\sum\limits_{\lambda \in \Lambda_+} \lambda & = & (\ell+1) +\binom{3\ell+6}{2}-\binom{\ell+3}{2}=4\ell^2+ 15\ell+13,\\
\sum\limits_{\lambda \in \Lambda_-} \lambda & =& \binom{\ell+1}{2} + (\ell+2) +\binom{4\ell+6}{2}-\binom{3\ell+6}{2} = 
4 \ell^2+7  \ell  +2.
\end{array}$$
Now, let $(i_1,\ldots,i_{2\ell+4})$ be any ordering of $\Lambda_+$ and $(j_1,\ldots,j_{2\ell+1})$ be any ordering of $\Lambda_-$.
Take the following $2\times n$ array with elements in $\Z_{2n+1}^\ast$:
$$A=\begin{array}{|c|c|c|c|c|c|}\hline
[+i_1]_{2n+1} & \cdots & [+i_{2\ell+4}]_{2n+1} & [-j_1]_{2n+1} & \cdots & [-j_{2\ell+1}]_{2n+1} \\\hline
[-i_1]_{2n+1} & \cdots & [-i_{2\ell+4}]_{2n+1} & [+j_1]_{2n+1} & \cdots & [+j_{2\ell+1}]_{2n+1} \\\hline
\end{array}.$$
Clearly, each column sums to $[0]_{2n+1}$. Furthermore, the elements of the first row of $A$ sum to 
$[4\ell^2+15\ell+13]_{2n+1}- [4\ell^2+7\ell+2]_{2n+1}=[8\ell+11]_{2n+1}=[0]_{2n+1}$.
This shows that $A$ is a $\Mo_{\Z_{2n+1}^\ast}(2,n)$, i.e., a $\MoS_{\Z_{2n+1}^\ast}(2,n;1)$.

Finally, we consider the case $c=2$. Set $N=4n+1$ and define the following blocks with elements in $\Z_N$,
where $x$ is a positive integer:
$$\begin{array}{rcl}
U_3 & =& \begin{array}{|c|c|c|}\hline
[1]_N & [3]_N & -[4]_N \\\hline
-[1]_N & -[3]_N & [4]_N \\\hline
     \end{array},\\ \\[-8pt]
V_3 & =& \begin{array}{|c|c|c|}\hline
[2]_N & [5]_N & -[7]_N \\\hline
-[2]_N & -[5]_N & [7]_N \\\hline
     \end{array}, 
               \end{array}$$
$$\begin{array}{rcl}
U_9 & =& \begin{array}{|c|c|c|c|c|c|c|c|c|}\hline
[1]_N & [2]_N & [3]_N & [4]_N & -[5]_N  & [12]_N & [13]_N & -[14]_N   & -[16]_N\\\hline
-[1]_N & -[2]_N & -[3]_N & -[4]_N & [5]_N  & -[12]_N & -[13]_N & [14]_N   & [16]_N\\\hline
     \end{array},\\ \\[-8pt]
V_9 & =& \begin{array}{|c|c|c|c|c|c|c|c|c|}\hline
[6]_N & [7]_N & [8]_N & [9]_N & [10]_N & [11]_N & -[15]_N & -[17]_N  & -[19]_N \\\hline
-[6]_N & -[7]_N & -[8]_N & -[9]_N & -[10]_N & -[11]_N & [15]_N & [17]_N  & [19]_N \\\hline
     \end{array},\\ \\[-8pt]
W_x & =& \begin{array}{|c|c|c|c|}\hline
 [x]_{N} & -[x+2]_{N} & - [x+3]_{N} & [x+5]_{N}\\\hline       
 -[x]_{N} & [x+2]_{N} &   [x+3]_{N} & -[x+5]_{N}\\\hline 
     \end{array}.
     \end{array}$$
Note that all these blocks have rows and columns that sum to $[0]_N$.

If $n=5$, we take 
$$\begin{array}{rcl}
 A_1 & =& \begin{array}{|c|c|c|c|c|}\hline
[1]_{21} & [2]_{21} & [13]_{21} & [16]_{21} & [10]_{21} \\ \hline
[20]_{21} & [19]_{21} & [8]_{21} & [5]_{21} & [11]_{21}\\ \hline
\end{array},\\ \\[-8pt]
 A_2 & =& \begin{array}{|c|c|c|c|c|}\hline
[3]_{21}  & [17]_{21}  & [6]_{21}  & [7 ]_{21}  & [9]_{21}  \\ \hline
[18]_{21}  & [4]_{21}  & [15]_{21}  & [14 ]_{21}  & [12]_{21}  \\ \hline
\end{array}.
\end{array}$$

If $n\equiv 3 \pmod 4$, write $n=4\ell+3$ and define $A_1$ as the $2\times n$ array obtained by the juxtaposition of $U_3$
and the blocks $W_{6+4j}$ with $j \in [0,\ell-1]$  (so, $A_1$ coincides with $U_3$ when $n=3$). Also, 
let $A_2$ be the $2\times n$ array defined by the juxtaposition of $V_3$ and the blocks $W_{n+3+4j}$ with $j \in [0,\ell-1]$. We have 
$\E(A_1)=\{\pm [z]_N \mid  z \in \Psi_1\}$ and $\E(A_2)=\{\pm [z]_N \mid  z \in \Psi_2\}$, where 
$$\begin{array}{rcl}
\Psi_1 &  =& \{1,3,4\}\cup \{6\} \cup  [8, n+2] \cup \{n+4\},\\
\Psi_2 & =& \{2, 5,7 \}\cup \{n+3\}\cup [n+5,2n-1] \cup \{2n+1\}.
\end{array}$$

If $n\equiv 1 \pmod 4$, with $n\geq 9$, write $n=4\ell+9$ and define $A_1$ as the $2\times n$ array obtained by the juxtaposition 
of $U_9$ and the blocks 
$W_{18+4j }$ with $j \in [0,\ell-1]$. Also, let $A_2$ be the $2\times n$ array defined by the juxtaposition of $V_9$ and 
the blocks $W_{n+9+4j}$ with $j \in [0,\ell-1]$. We have 
$\E(A_1)=\{\pm [z]_N \mid  z \in \Psi_1\}$ and $\E(A_2)=\{\pm [z]_N \mid  z \in \Psi_2\}$, where 
$$\begin{array}{rcl}
\Psi_1 &= & \{1, 2, 3, 4, 5, 12, 13, 14,16 \}\cup \{ 18\} \cup [20, n+8 ] \cup \{n+10\}, \\
\Psi_2 &= & \{6, 7, 8, 9, 10,  11, 15, 17, 19\}\cup \{n+9\}\cup [n+11,2n-1]\cup \{2n+1\}.
\end{array}$$

In all the previous cases, we get $\E(A_1)\cup \E(A_2)=\{\pm [z]_{4n+1}\mid z\in [1,2n]\}=\Z_{4n+1}^*$, and hence the set $\{A_1,A_2\}$ is a $\MoS_{\Z_{4n+1}^*}(2,n;2)$.
\end{proof}

For instance, we can construct a $\Mo_{\Z_{19}^*}(2,9)$ following the proof of the previous proposition.
In this case,  $\Lambda_+=\{2, 4,5,6,7,8 \}$ and $\Lambda_-=\{1,3,9\}$:
$$\begin{array}{|c|c|c|c|c|c|c|c|c|}\hline
[2]_{19} &  [4]_{19} & [5]_{19} & [6]_{19} & [7]_{19} & [8]_{19} & [18]_{19}  & [16]_{19} & [10]_{19} \\\hline
[17]_{19} & [15]_{19} & [14]_{19} & [13]_{19} & [12]_{19} & [11]_{19} &  [1]_{19} & [3]_{19} & [9]_{19} \\\hline
\end{array}.$$

\begin{cor}\label{cor pieno}
Let  $m,n\geq 2$. There exists a $\Mo_{\Z_{mn+1}^\ast}(m,n)$ if and only if $mn+1$ is odd and greater than $5$.
\end{cor}

We now consider the general case of any finite abelian group, not necessarily cyclic.

\begin{lem}\label{noncic}
Let  $\Gamma_1,\Gamma_2$ be two finite abelian groups  with $|\Gamma_1|=2a+1$ and $|\Gamma_2|=2b+1$.
Suppose that $2a+1$ divides $2b+1$ and that there exists a $\Mo_{\Gamma_2^*}(2,b)$.
Then, there exists a  $\Mo_{\Gamma^*}(2,a+b+2ab)$, where $\Gamma=\Gamma_1\oplus \Gamma_2$.
\end{lem}

\begin{proof}
Note that $\Gamma_1,\Gamma_2$ belong to $\G$, since they both have odd order.
Let $(g_1,\ldots,g_{2a})$ be any ordering of the elements of $\Gamma_1^\ast$ such that $g_{a+\ell}=-g_{\ell}$ 
for all $\ell \in [1,a]$, and
let $(h_1,\ldots,h_{2b+1})$ be any ordering of the elements of $\Gamma_2$.
Let $V=(v_{i,j})$  be a $\Mo_{\Gamma_2^*}(2,b)$. We construct a $2\times (a+b+2ab)$ array $C$ as follows. 

First, take the $2\times b$ array
$T= \left( t_{i,j}\right)$, where $t_{i,j}=(0_{\Gamma_1}, v_{i,j})$ for any
$i=1,2$ and any $j\in [1,b]$.
Since the elements of each row and each column of $V$ sum to $0_{\Gamma_2}$, 
the elements of each row and each column of $T$ sum to $(0_{\Gamma_1}, 0_{\Gamma_2})$. 
Now, for every  $\ell\in [1,a]$, we construct a $2\times (2b+1)$ array $U_\ell = \left(u^{(\ell)}_{i,r} \right)$  setting
$u^{(\ell)}_{1,r}=(g_\ell, h_r)$ and $u^{(\ell)}_{2,r}=(- g_\ell, - h_r )$ for every $r\in [1,2b+1]$.
Note that the elements of each column of $U_\ell$ sum to $(0_{\Gamma_1}, 0_{\Gamma_2})$, while
the elements of each row sum to $\pm \left(|\Gamma_2| g_\ell, \sum\limits_{r=1}^{2b+1} h_r \right)
=(\pm|\Gamma_2| g_\ell, 0_{\Gamma_2} )$.
Since the order of $\Gamma_1$ divides the order of $\Gamma_2$, applying Lagrange's theorem, we obtain $|\Gamma_2| g_\ell=0_{\Gamma_1}$, 
whence $(\pm|\Gamma_2| g_\ell, 0_{\Gamma_2} )= (0_{\Gamma_1}, 0_{\Gamma_2})$.
Finally, we construct the array $C$:
$$C=\;\begin{array}{|c|c|c|c|}\hline
T & U_1 & \cdots & U_a \\\hline  
  \end{array}.$$
By the previous observations, the elements of each row and each column of $C$ sum to $0_{\Gamma}$.
Furthermore, $\E(T)=\{(0,h) \mid h \in \Gamma_2^\ast\}$ and $\E(U_\ell)=\{\pm (g_\ell, h) \mid h \in \Gamma_2\}$.
Hence, $\E(C)=\E(T)\cup \left(\bigcup\limits_{\ell=1}^{a}\E(U_\ell)\right)=\Gamma^\ast$.
This proves that $C$ is a  $\Mo_{\Gamma^*}(2,a+b+2ab)$.
\end{proof}

\begin{cor}\label{corG}
Let $\Gamma$ be an abelian group of order $2n+1\geq 3$. There exists a $\Mo_{\Gamma^*}(2,n)$ if and only if 
$\Gamma \not\in \{\Z_3,\Z_5,\Z_3\oplus \Z_3\}$.
\end{cor}

\begin{proof}
First, decompose $\Gamma$ as the direct sum $\Z_{2n_1+1}\oplus \ldots\oplus \Z_{2n_r+1}$ of cyclic groups of order $2n_i+1$ such that
$2n_i+1$ divides $2n_{i+1}+1$ for all $i\in [1,r-1]$.
Suppose  $2n_r+1\geq 7$. In this case, there exists a $\Mo_{\Z_{2n_r+1}^*}(2,n_r)$ by Proposition \ref{ciclico}.
If $r=1$, our proof is complete. If $r\geq 2$, applying Lemma  \ref{noncic}, there also exists a 
$\Mo_{\Gamma_2^*}\left(2, \frac{(2n_{r}+1)(2n_{r-1}+1)-1}{2} \right)$, where $\Gamma_2=\Z_{2n_{r-1}+1}\oplus \Z_{2n_r+1}$.
Again, if $r=2$, our proof is complete; otherwise, we apply repeatedly Lemma \ref{noncic}, obtaining at the end 
a $\Mo_{\Gamma^*}(2, n )$.

So, we are left to consider the cases when $\Gamma$ is an elementary abelian group of exponent $2n_r+1\in \{3, 5\}$.
We have already seen that there is no $\Mo_{\Z_3^*}(2,1)$ and no $\Mo_{\Z_5^*}(2,2)$. Also,
it is not difficult to prove that there is no $\Mo_{\Gamma^*}(2,4)$ when $\Gamma=\Z_3\oplus \Z_3$.
Assume $2n_r+1=3$ and $r\geq 3$. We first take the following $\Mo_{(\Z_3\oplus \Z_3\oplus \Z_3)^\ast}(2,13)$, 
where the entry $xyz$ must be read as 
$([x]_3,[y]_3,[z]_3)$:
$$\begin{array}{|c|c|c|c|c|c|c|c|c|c|c|c|c|}\hline
001 & 010 & 011 & 021 & 100 & 101 & 102 & 110 & 111 & 221 & 210 & 212 & 122\\\hline
002 & 020 & 022 & 012 & 200 & 202 & 201 & 220 & 222 & 112 & 120 & 121 & 211\\\hline
\end{array}.$$ 
Then, we proceed as before applying repeatedly Lemma \ref{noncic} until we obtain a  $\Mo_{\Gamma^*}(2, n )$.
Assume $2n_r+1=5$ and $r\geq 2$. In this case, it suffices to apply the previous argument starting with the following
$\Mo_{(\Z_5\oplus \Z_5)^\ast}(2,12)$,
where the entry $xy$ must be read as $([x]_5,[y]_5)$:
$$\begin{array}{|c|c|c|c|c|c|c|c|c|c|c|c|}\hline
01 & 02 & 10 & 11 & 12 & 13 & 41 & 20 & 21 & 33 & 32 & 24\\\hline
04 & 03 & 40 & 44 & 43 & 42 & 14 & 30 & 34 & 22 & 23 & 31\\\hline
\end{array}.$$
\end{proof}

\section{Conclusions}

In this paper we introduced the concepts of magic and zero-sum magic partially filled array whose elements belong to a subset $\Omega$ 
of an abelian group $\Gamma$. We think these arrays are worth to be studied not only because they generalize well known objects such as
magic rectangles, $\Gamma$-magic rectangles, signed magic arrays, integer/non integer/relative Heffter arrays,
but also because of their connection with $\Gamma$-supermagic labelings and with zero-sum $\Gamma$-magic graphs.

One of our main achievements is the complete solution for the existence of magic rectangle sets with empty cells
$\MRS(m,n;s,k;c)$ (Theorem \ref{main2}),  which
extends Froncek's result about tight magic rectangle sets.
We have also investigated two significant cases: when $\Omega$ is the full group $\Gamma$ and when $\Omega$ is the set $\Gamma^*$ of 
nonzero elements of $\Gamma$. The  main results about these two cases are described in Theorem \ref{main}, that we can now prove.

\begin{proof}[Proof of Theorem \ref{main}]
Items (1), (5) and (6) follow from Corollaries \ref{cyc}, \ref{cor pieno} and \ref{corG};
item (4) follows from Proposition \ref{ciclico}.
To prove (2) and (3), set $d=\gcd(s,k)$. 
If $d\equiv 0 \pmod 4$, there exists a diagonal $\MoS_\Gamma\left(\frac{nk}{d}; d; c\right)$ 
for every $\Gamma \in \G$ of order $nkc$, by Proposition \ref{d0}.
If $nk$ is odd and $\gcd( n, d - 1 )=1$, then $d\geq 3$ is odd: so, there exists a diagonal 
$\Mo_{\Z_{d}\oplus \Z_{nk/d}}\left(\frac{nk}{d}; d\right)$ by Proposition \ref{diag_k}.
In both cases, we apply Theorem \ref{sq->rt}.
\end{proof}

\end{document}